\documentclass[11pt,reqno]{amsart}
\usepackage{amsmath,amsfonts,epsfig,amssymb,graphics,enumerate,color,ulem}
 \usepackage{hyperref}
\newtheorem{thm}{Theorem}[section]

\newtheorem{cor}{Corollary}[section]
\newtheorem{prop}{Proposition}[section]
\newtheorem{lemma}{Lemma}[section]
\newtheorem{defn}{Definition}[section]
\newtheorem{rem}{Remark}[section]

\newcommand{\N}{\mathbb{N}}
\newcommand{\R}{\mathbb{R}}

\newcommand{\T}{\mathcal{T}}
\newcommand{\A}{\mathcal{A}}

\begin{document}
\title[Eigenvalues of perturbed Laplace operators]{Eigenvalues of perturbed Laplace operators on compact manifolds}
\author{Asma Hassannezhad}
\address{Universit\'e de Neuch\^atel, Institut de Math\'ematiques, Rue Emile-Argand 11, Case postale 158, 2009 Neuch\^atel,
Switzerland}
\email{asma.hassannezhad@unine.ch}
\thanks{The author has benefited from the support of the \textit{boursier du gouvernement Fran\c cais} during her stay in Tours.}

\date{}
\begin{abstract}
We obtain upper bounds for the eigenvalues of the Schr\"odinger operator $L=\Delta_g+q$ depending on integral quantities of the potential $q$ and a conformal invariant called the \textit{min-conformal volume}. Moreover, when the Schr\"odinger operator $L$ is positive, integral quantities of $q$ which appear in upper bounds, can be replaced by  the mean value of the potential $q$. The upper bounds we obtain are compatible with the asymptotic behavior of the eigenvalues. We also obtain upper bounds for the eigenvalues of the weighted Laplacian or the Bakry--\'Emery Laplacian $\Delta_\phi=\Delta_g+\nabla_g\phi\cdot\nabla_g$ using two approaches: First, we use the fact that $\Delta_\phi$ is unitarily equivalent to a Schr\"odinger operator and we get an upper bound in terms of the $L^2$-norm of $\nabla_g\phi$ and the min-conformal volume. Second, we use its variational characterization and we obtain upper bounds in terms of the $L^\infty$-norm of $\nabla_g\phi$ and a new conformal invariant. The second approach leads to  a Buser type upper bound and also gives upper bounds which do not depend on $\phi$ when the Bakry-\'Emery Ricci curvature is non-negative.
\end{abstract}
%%%%%%%%%%%%%%%%%%%%%%%%%%%%%%%%%%%%%%%%%%%%%%%%%%%%%%%%%%%%%%%%%%%%%%%%%%%%%%%%%%%%%%%%%%%%%%%%%%%%%%
\subjclass[2010]{
58J50, 35P15, 47A75
}
\keywords{Schr\"odinger operator, Bakry--\'Emery Laplace operator, eigenvalue, upper bound, conformal invariant}

%%%%%%%%%%%%%%%%%%%%%%%%%%%%%%%%%%%%%%%%%%%%%%%%%%%%%%%%%%%%%%%%%%%%%%%%%%%%%%%%%%%%%%%%%%%%%%%%%%%%%%
\maketitle
%\dosecttoc

\tableofcontents

\section{Introduction and statement of the results}
In this paper, we study upper bound estimates for the eigenvalues of Schr\"odinger operators and weighted Laplace operators or Bakry--\'Emery Laplace operators.\\
\noindent\textit{\large Schr\"odinger Operator.~}
Let $(M,g)$ be a compact Riemannian manifold of dimension $m$
and $q\in C^0(M)$. The eigenvalues  of the Schr\"odinger
operator ${L:=\Delta_g+q}$ acting on functions constitute a non-decreasing, semi-bounded sequence of real numbers going to infinity.
\[\lambda_1{(\Delta_g+q)}\leq\lambda_2{(\Delta_g+q)}\leq\cdots\leq\lambda_k{(\Delta_g+q)}\leq\cdots\nearrow\infty.\]
The well-known Weyl law which describes the the asymptotic behavior of the eigenvalues of the Laplacian \cite{Be} can be easily extended to the eigenvalues of   Schr\"odinger operators on compact Riemannian manifolds:
%The asymptotic behavior of $\lambda_k(\Delta_g+q)$ is given by:
\begin{equation}\label{weylsch}
 \lim_{k\to\infty}\lambda_k(\Delta_g+q)\left(\frac{\mu_g(M)}{k}\right)^{\frac{2}{m}}=
\alpha_m,
\end{equation}
\noindent where  $\alpha_m=4\pi^2\omega_m^{-\frac{2}{m}}$  and $\omega_m$ is
the volume
of the unit ball in  $\mathbb{R}^m$.

It describes that normalized eigenvalues, $\lambda_k(\Delta_g+q)\left(\frac{\mu_g(M)}{k}\right)^{\frac{2}{m}}$,
asymptotically tend to a constant depending only on the dimension. However, upper bounds of  normalized eigenvalues  in general cannot be independent of   geometric invariants and  the potential $q$  (see  \cite{CD} or the introduction  of \cite{H}). We shall obtain upper bounds for normalized eigenvalues  depending on some geometric invariants and  integral quantities of the potential $q$. Moreover, these upper bounds are compatible with the asymptotic behavior in \eqref{weylsch}  i.e.  they tend asymptotically to a constant depending only on the dimension as $k$ goes to infinity. 

 %Our main concern is to obtain upper bounds for the   eigenvalues of the Schr\"odinger
%operator $L$  which are compatible with the above asymptotic behavior.  We shall obtain upper bounds for normalized eigenvalues, $\lambda_k(\Delta_g+q)\left(\frac{\mu_g(M)}{k}\right)^{\frac{2}{m}}$,  depending on some geometric invariants and  integral quantities of the potential $q$. Moreover, these upper bounds are compatible with the asymptotic behavior in \eqref{weylsch}  i.e.  their limits are constants depending only on the dimension as $k$ goes to infinity.  We note that upper bounds of  normalized eigenvalues  in general cannot be independent of   geometric invariants and  the potential $q$  (see  \cite{CD} or the introduction  of \cite{H}).

Numerous articles are devoted to study how the eigenvalues of $L$ can be controlled in terms of geometric invariants of the manifold and  quantities depending on the potential. 
From the variational characterization of eigenvalues, it is easy to see that
\[\lambda_1{(\Delta_g+q)}\leq\frac{1}{\mu_g(M)}\int_Mqd\mu_g.\]
For the second eigenvalue $\lambda_2{(\Delta_g+q)}$, an upper bound in
terms of the mean value  of the potential $q$ and  a conformal
invariant was obtained  by El Soufi and  Ilias \cite[Theorem
2.2]{EI}:
\begin{equation}\label{cn}
\lambda_2{(\Delta_g+q)}\leq
m\left(\frac{V_c([g])}{\mu_g(M)}\right)^{\frac{2}{m}}+\frac{\int_Mqd\mu_g}{\mu_g(M)},
\end{equation}
where $V_c([g])$ is the conformal volume defined by Li
and
Yau  \cite{LY2}  which only depends on the conformal class  of $g$, denoted $[g]$.

For a compact orientable  Riemannian surface $(\Sigma_\gamma,g)$  of
genus $\gamma$, they obtained the following inequality as a
consequence of Inequality (\ref{cn}):
\begin{equation}
\lambda_2{(\Delta_g+q)}\leq
\frac{8\pi}{\mu_g(\Sigma_\gamma)}\left[\frac{\gamma+3}{2}\right]+\frac{\int_{\Sigma_\gamma}qd\mu_g}{\mu_g(\Sigma_\gamma)},
\end{equation}
where $[\frac{\gamma+3}{2}]$ is the integer part of
$\frac{\gamma+3}{2}$.\

 For higher eigenvalues of  Schr\"odinger operators, Grigor'yan, Netrusov and Yau
\cite{GNY}  proved a general and abstract result that can be stated in the case of  Schr\"odinger operators
as follows: {Given positive constants $N$ and $C_0$, assume that a
compact Riemannian manifold $(M,g)$ has the $(2,N)$-covering
property (i.e. each ball of radius $r$ can be covered by
$N$ balls of radius $r/2$) and $\mu_{g}(B(x,r))\leq C_0r^2$ for
every {$x\in M$ and every } $r>0$. Then for {every} $q\in C^0(M)$
we have \cite[Theorem 1.2 (1.14)]{GNY}:}
\begin{equation}\label{delta}
\lambda_k{(\Delta_g+q)}\leq\frac{Ck+\delta^{-1} \int_Mq^+d\mu_g-\delta
\int_Mq^-d\mu_g}{\mu_g(M)},
\end{equation}
where $\delta\in (0,1)$ is a constant which  depends  only on $N$,  $C>0$ is a constant which depends on
$N$ and $C_0$, {and $q^{\pm}=\max\{|\pm q|,0\}$}. \\ Moreover, if $L$ is a positive operator
\cite[Theorem 5.15]{GNY}, then
\begin{equation}\label{deltach3}
\lambda_k{(\Delta_g+q)}\leq\frac{Ck+\int_Mqd\mu_g}{\epsilon \mu_g(M)},
\end{equation}
where $\epsilon\in(0,1)$ depends only on $N$ and $C$ depends on
$N$ and $C_0$.\\
%\sout{\bl{ In \cite{GNY}, it is conjectured that
%Inequality (\ref{deltach3}) is also true without assuming $L$ to
%be positive.}}

The above inequalities in dimension 2 have special feature as
follows. Let $\Sigma_\gamma$ be a compact orientable Riemannian surface of genus $\gamma$.
Then for every Riemannian metric $g$ on $\Sigma_\gamma$ and every
$q\in C^0(\Sigma_\gamma)$ we have \cite[Theorem 5.4]{GNY}:
\begin{equation*}\label{delta1}
\lambda_k{(\Delta_g+q)}\leq\frac{Q(\gamma+1)k+\delta^{-1}
\int_{\Sigma_\gamma}q^+d\mu_g-\delta \int_{\Sigma_\gamma}q^-d\mu_g}{\mu_g(\Sigma_\gamma)},
\end{equation*}
where $\delta\in (0,1)$ and $Q>0$ are absolute constants.\

Inequalities {(\ref{delta}) and  (\ref{deltach3})}  are not compatible with the asymptotic behavior regarding to the power of $k$, except in dimension 2. Yet, for surfaces, the limit of the above  upper bound  for normalized eigenvalues depends on the genus $\gamma$ as $k$ goes to infinity. Therefore, it is not compatible with \eqref{weylsch}.\

We obtain upper bounds  which generalize and
improve the above inequalities {without imposing any condition on the metric} and which are
compatible with the asymptotic behavior. Before stating our theorem, we need to recall the definition of  the \textit{min-conformal volume}. For a compact Riemannian manifold $(M,g)$,  its \textit{min-conformal volume} is defined as follows \cite{H}.
  \[
V([g])=\inf\{\mu_{g_{0}}(M)~:~g_{0}\in[g],~{\rm
Ricci}_{g_{0}}\geq-(m-1)\}.
\]
\begin{thm}\label{thm}{ There exist positive constants $\alpha_m\in(0,1)$,  $B_m$ and $C_m$ depending only on $m$ such that
 for every compact $m$-dimensional Riemannian manifold $(M,g)$, {every potential $q\in C^0(M)$}  and  every $k\in \N^*$, we have}
\begin{align}\label{222}
\nonumber \lambda_k{(\Delta_g+q)}&\leq \frac{\alpha_m^{-1} \int_Mq^+d\mu_g-\alpha_m
\int_Mq^-d\mu_g}{\mu_g(M)}\\
&+B_m\left(\frac{V([g])}{\mu_g(M)}\right)^{\frac{2}{m}}+
C_m\left(\frac{k}{\mu_g(M)}\right)^{\frac{2}{m}},
\end{align}
In particular, when the potential $q$ is {nonnegative} one has
%\noindent It deduces  from the fact $\lambda_k^{\Delta_g+q}(m,g)\leq\lambda_K^{\Delta_g+|q|}$, the following inequality
\begin{equation}\label{pospot3}\lambda_k{(\Delta_g+q)}\leq A_m\frac{\int_Mqd\mu_g}{\mu_g(M)}+B_m\left(\frac{V([g])}{\mu_g(M)}\right)^{\frac{2}{m}}+
C_m\left(\frac{k}{\mu_g(M)}\right)^{\frac{2}{m}},\end{equation}
where $A_m=\alpha_m^{-1}$.\
\end{thm}
We also obtain  upper bounds for eigenvalues of positive Schr\"odinger operators. Note that  the positivity of the Schr\"odinger operator $L=\Delta_g+q$ implies that $\int_Mq\geq0$ and  $q$ here may not be nonnegative.  The following upper bound generalizes  Inequalities (\ref{deltach3}) and \eqref{pospot3}.
\begin{thm}\label{pot}{ There exist constants $A_m>1$, $B_m$ and $C_m$
depending only on $m$ such that if $L=\Delta_g+q$, $q\in C^0(M)$ is a positive
 operator then for every compact $m$-dimensional Riemannian manifold $(M^m,g)$ and every $k\in \N^*$ we have }
\begin{equation*}
 \lambda_k{(\Delta_g+q)}\leq A_m\frac{\int_M qd\mu_g}{\mu_g(M)}+B_m\left(\frac{V([g])}{\mu_g(M)}\right)^{\frac{2}{m}}+
C_m\left(\frac{k}{\mu_g(M)}\right)^{\frac{2}{m}}.
\end{equation*}
\end{thm}
Given the Schr\"odinger operator $L=\Delta_g+q$, for every $\varepsilon>0$, the Schr\"odinger operator $\tilde{L}=\Delta_g+q-\lambda_1{(L)}+\varepsilon$ is positive and $\lambda_k{(\tilde{L})}=\lambda_k{(L)}-\lambda_1{(L)}+\varepsilon$. When $\varepsilon$ goes to zero, Theorem \ref{thm} leads to the following:
\begin{cor}Under the assumptions of Theorem \ref{thm} we get
\begin{align*}
\lambda_k{(\Delta_g+q)}&\leq A_m\frac{\int_Mqd\mu_g}{\mu_g(M)}+(1-A_m)\lambda_1{(\Delta_g+q)}\\
&+B_m\left(\frac{V([g])}{\mu_g(M)}\right)^{\frac{2}{m}}+
C_m\left(\frac{k}{\mu_g(M)}\right)^{\frac{2}{m}}.
\end{align*}
\end{cor}
In the 2-dimensional case, for a compact orientable  Riemannian surface $(\Sigma_\gamma,g)$ of
genus $\gamma$, thanks to the uniformization  and Gauss-Bonnet
theorems, one has $V([g])\le 4\pi\gamma $. Therefore, in compact orientable Riemannian surfaces, one can replace  the min-conformal volume  by the topological invariant $4\pi \gamma$ in the above inequalities.
\begin{cor}\label{baksur}There exist  absolute constants $a\in(0,1)$, $A$ and $B$ such that, for every
compact orientable Riemannian surface $(\Sigma_\gamma,g)$  of genus $\gamma $, {every potential $q\in C^0(M)$}
and every $k\in \N^*$, we have
\begin{equation}\label{3}
\lambda_k{(\Delta_g+q)}\mu_g(\Sigma_\gamma)\leq
\int_{\Sigma_\gamma}(aq^+-a^{-1}q^-)d\mu_g+A
\gamma+Bk.
\end{equation}
And if $L$ is a positive operator then
\begin{equation*}
\lambda_k{(\Delta_g+q)}\mu_g(\Sigma_\gamma)\leq
a\int_{\Sigma_\gamma}qd\mu_g+A
\gamma+Bk.
\end{equation*}
\end{cor}
An interesting application of Theorem \ref{thm} is  the case  of weighted Laplace operators or Bakry--\'Emery Laplace operators.\\

\noindent \textit{\large Bakry--\'Emery Laplacian.~} Let   $(M,g)$ be a Riemannian manifold and $\phi\in C^2(M)$. The corresponding weighted Laplace operator $\triangle_\phi$ is defined as follows.
\[\Delta_\phi=\Delta_g+\nabla_g\phi\cdot\nabla_g.\]
This operator is associated with the quadratic functional $\int_M|\nabla_g f|^2e^{-\phi}d\mu_g$  i.e.
\[
\int_M\Delta_\phi f{h}e^{-\phi}d\mu_g=\int_M\langle \nabla_g f,\nabla_g {h}\rangle e^{-\phi}d\mu_g.\]
This operator is an elliptic operator on $C^\infty_c(M)\subseteq L^2(e^{-\phi} d\mu_g)$  and can be extended to a self-adjoint operator with the weighted measure $e^{-\phi}d\mu_g$. In this sense, it arises as a generalization of the Laplacian. The weighted Laplace operator $\Delta_\phi$ is also known  as the diffusion operator or the Bakry--\'Emery Laplace operator which is used to study the diffusion process (see for instance, the pioneering work of Bakry and \'Emery \cite{BaE},  the paper of  Lott \cite{lot}, and Lott and Villani \cite{LV} on this topic). The triple $(M,g,\phi)$ is called a Bakry--\'Emery manifold  where $\phi\in C^2(M)$ and $(M,g)$ is a Riemannian manifold with  the weighted measure $e^{-\phi}d\mu_g$ (see \cite{LR}, \cite{Ro}).
The interplay between geometry of $M$ and the behavior of $\phi$ is mostly taken into account by means of new notion of curvature called
	the Bakry--\'Emery Ricci tensor\footnote{ The
Bakry--\'Emery Ricci tensor ${\rm Ricci}_\phi$ is also referred to as
the $\infty$-Bakry--\'Emery Ricci tensor.
We will denote
${\rm Ricci}_\phi$ and ${\rm Hess} \phi$ by ${\rm Ricci}_\phi(M,g)$ and
${\rm Hess}_g\phi$ wherever any confusion might occur.}  that is defined as follows
\[\label{bakryricci}{\rm Ricci}_\phi={\rm Ricci}_g+{\rm Hess} \phi.\]
Our aim is to find upper bounds for the eigenvalues of $\Delta_\phi$  denoted by $\lambda_k{(\Delta_\phi)}$ in terms of the geometry of $M$ and of properties of $\phi$.\\
{Upper bounds for the first  eigenvalue $\lambda_1{(\Delta_\phi)}$ of complete non-compact Riemannian manifolds have been recently considered in several works (see \cite{MW},\cite{Se},\cite{SZ}, \cite{Wu} and \cite{Wu2}). These upper bounds depend  on the $L^\infty$-norm of  $\nabla_g\phi$ and a lower bound of the Bakry--\'Emery Ricci tensor:\\
 Let $(M,g,\phi)$ be a {complete} non-compact Bakry--\'Emery manifold of dimension $m$ with ${\rm Ricci}_\phi\geq-\kappa^2(m-1)$ and $|\nabla_g\phi|\leq\sigma$
for some constants $\kappa\geq0$ and $\sigma>0$. Then we have \cite[Proposition 2.1]{SZ} (see also \cite{MW}, \cite{Wu} and \cite{Wu2}):
\begin{equation}\label{bakupp}
\lambda_1{(\Delta_\phi)}\leq\frac{1}{4}((m-1)\kappa+\sigma)^2.
\end{equation}
In particular, if ${\rm Ricci}_\phi\geq0$, then we have
\begin{equation}\label{bakup}
\lambda_1
{(\Delta_\phi)}\leq\frac{1}{4}\sigma^2.
\end{equation}}
We consider compact Bakry--\'Emery manifolds and we present two approach{es} to obtain upper bounds for the eigenvalues of the Bakry--\'Emery Laplace operator in terms of the geometry of $M$ and of the properties of $\phi$.\

 {\bf First approach.} One {can} see that $\Delta_\phi$ is unitarily equivalent to the Schr\"odinger operator $L=\Delta_g+\frac{1}{2}\Delta_g\phi+\frac{1}{4}|\nabla_g\phi|^2$ (see for example \cite[page 28]{Se}). Therefore, as a consequence of Theorem \ref{pot} we obtain an upper bound for $\lambda_k(\Delta_\phi)$ in terms of the min-conformal volume  and the  $L^2$-norm of $\nabla_g\phi$.
%
%           theorem
%
\begin{thm}\label{bakint}
There exist constants $A_m$, $B_m$ and $C_m$ depending on $m\in\N^*$, such that for every $m$-dimensional compact Bakry--\'Emery  manifold $(M,g,\phi)$,  we have
\[\lambda_k{(\Delta_\phi)}\leq A_m\frac{1}{\mu_g(M)}\|\nabla_g\phi\|^2_{L^2(M)}+B_m\left(\frac{V([g])}{\mu_g(M)}\right)^{\frac{2}{m}}+
C_m\left(\frac{k}{\mu_g(M)}\right)^{\frac{2}{m}}.\]
\end{thm}
It is worth noticing that in full generality, it is not possible to obtain upper bounds which do not depend on $\phi$ (see for instance
\cite[Section 2]{SZ}). However, we will see that for compact manifolds with {nonnegative Bakry--\'Emery Ricci curvature} we can find upper bounds which do not depend on $\phi$ (see Corollary \ref{bakpos} below). \\
In the 2-dimensional case, as a result of Corollary \ref{baksur} we obtain
\begin{cor} There exist  absolute constants $a\in(0,1)$, $A$ and $B$ such that, for every
compact orientable Riemannian surface $(\Sigma_\gamma,g)$  of genus $\gamma $ and every $k\in \N^*$, we have
\begin{equation*}
\lambda_k(\Delta_\phi)\mu_g(\Sigma_\gamma)\leq
a\|\nabla_g\phi\|^2_{L^2(\Sigma_\gamma)}+A
\gamma+Bk.
\end{equation*}
\end{cor}
{\bf Second approch.} It is based on using the technique introduced in \cite{H}  which was successfully applied for the Laplace operator $\Delta_g $ on  Riemannian manifolds in \cite[Theorem 1.1]{H}.  We obtain upper bounds for eigenvalues of $\Delta_\phi$ in terms of a conformal invariant. We also obtain a Buser type upper bound for $\lambda_k{(\Delta_\phi)}$ (see below Corollary \ref{bakbuser}).
\begin{defn}\label{bakcon1}Let $(M,g,\phi)$ be {a compact Bakry}--\'Emery
manifold. We define
the $\phi-$min conformal volume as
{\begin{equation}\label{bakcon} V_\phi([g
])=\inf\{\mu_\phi(M,g_0): g_0\in[g], {\rm Ricci}_\phi(M,g_0)\geq-(m-1)\},
\end{equation}}
where $\mu_\phi(M,g_0)$ is the weighted measure\footnote{For a Bakry--\'Emery manifold $(M,g,\phi)$, when $\mu_\phi$ is the weighted measure with respect to the metric $g$, we simply denote the weighted measure of a measurable  subset $A$ of $M$ by $\mu_\phi(A)$ instead of $\mu_\phi(A,g)$.} of $M$ with respect to the metric $g_0$.
\end{defn}
{Note that up to dilations\footnote{Notice  ${\rm Hess}\phi$ and ${\rm Ricci}_g$ do not change
under dilations. If ${\rm Ricci}_\phi(M,g)\geq-\kappa^2(m-1)g$,  then $\forall \alpha>0$, ${\rm Ricci}_\phi(M,g_0):={\rm Ricci}_\phi(M,\alpha g)={\rm Ricci}_\phi(M,g)\geq-\kappa^2(m-1)g=-\frac{\kappa^2}{\alpha}(m-1)g_0.$} there is always a Riemannian metric $g_0\in[g]$ such that
${\rm Ricci}_\phi(M,g_0)\geq-(m-1)$.}
%\[{\rm Ricci}_\phi(M,g_0)={\rm Ricci}_{g_0}+{\rm Hess}_{g_0}\phi\geq-(m-1)+{\rm Hess}_{g}\phi\geq-(1+\alpha^2)(m-1).\]
%Before stating our result we need to introduce one more notation. Let
%$x\in M$ and $\phi\in C^2(M)$. Using
%geodesic polar coordinates centered at $x$, we denote the radial
%derivative of $\phi$ as $\partial_r\phi$. The inequality $\partial_r\phi>-\sigma$  means that for every $x\in M$ the corresponding radial gradient of $\phi$ is at least equal to $-\sigma$, where $\sigma\geq0$.\\
 We are now ready to state our  theorem.
\begin{thm}\label{bakmain}{
There exist positive constants $A(m)$
and $B(m)$  depending only on $m\in\N^*$ such that for every {compact} Bakry--\'Emery manifold $(M,g,\phi)$  with
$|\nabla_g\phi|\leq\sigma$ {for some  $\sigma\geq0$} and for every $k\in\N^*$, we have
\begin{equation}\label{bakineq}
\lambda_k{(\Delta_\phi)}\leq
A(m)\max\{\sigma^2,1\}\left(\frac{V_\phi([g
])}{\mu_\phi(M)}\right)^{\frac{2}{m}}+B(m)\left(\frac{k}{\mu_\phi(M)}\right)^{\frac{2}{m}}.
\end{equation}}
\end{thm}
{If a metric $g$ is conformally equivalent to a metric $g_0$ with ${\rm Ricci}_\phi(M,g_0)\geq0$,  then $V_\phi([g])=0$. Therefore, an immediate consequence of Theorem \ref{bakmain} is the following.}
\begin{cor}\label{bakpos}{ There exists a positive constant $A(m)$ which depends
only on $m\in\N^*$ such that
for every {compact} Bakry--\'Emery manifold $(M,g,\phi)$  with
$V_\phi([g])=0$,  and for every $k\in \N^*$
\begin{equation}
\lambda_k{(\Delta_\phi)}\leq
A(m)\left(\frac{k}{\mu_\phi(M)}\right)^{\frac{2}{m}}.
\end{equation}}
\end{cor}
The above upper bound is similar to the upper bound for the eigenvalues of the Laplacian in  Riemannian manifolds $(M,g)$ when $V([g])=0$ (see \cite{Ko}).  \\
If ${\rm Ricci}_\phi(M)>-\kappa^2(m-1)$ for some $\kappa\geq0$, then for $g_0=\kappa^2g$ one has  ${\rm Ricci}_\phi(M,g_0)>-(m-1)$ and $V_\phi([g])\leq\mu_\phi(M,g_0)=\kappa^m\mu_\phi(M,g)$. Replacing in Inequality \eqref{bakineq}, we get a Buser type upper bound for the eigenvalues of the Bakry--\'Emery Laplacian.
\begin{cor}[Buser type upper bound]\label{bakbuser} There are positive constants $A(m)$
and $B(m)$ depending only on $m\in\N^*$ such that for every compact Bakry--\'Emery manifold $(M,g,\phi)$ with
${\rm Ricci}_\phi(M)>-\kappa^2(m-1)$ and
$|\nabla_g\phi|\leq\sigma$ for some $\kappa\geq0$ and $\sigma\geq0$, and   for every $k\in \N^*$, we have
\begin{equation*}
\lambda_k{(\Delta_\phi)}\leq A(m)\max\{\sigma^2,1\}\kappa^2+
B(m)\left(\frac{k}{\mu_\phi(M)}\right)^{\frac{2}{m}}.
\end{equation*}
\end{cor}
%The above corollary is also a generalization of Inequality (\ref{bakupp}) to higher eigenvalues {(see also Theorem \ref{bakbuserapp})}.
 A weaker version of Corollary \ref{bakbuser} can be proved directly  {by} the classic idea {used by Buser \cite{Bu2}, Li and Yau \cite{LY}}. We refer the reader to Appendix \ref{dovomii} where we give a simple direct proof.
{\begin{rem}
Notice that all of the results have been mentioned above for compact manifolds are also valid when bounded sudomains of complete manifolds  with the Neumann boundary condition are considered.
\end{rem}}
%This paper is organized as follows. In Section \ref{sec2p} we study the eigenvalues of Schr\"odinger operators. Section \ref{sec3p} concerns the %eigenvalues of the Bakry--\'Emery Laplacian.
%
\section*{Acknowledgements}
This paper is part of the author's PhD thesis under the direction
of Professors Bruno Colbois (Neuch\^atel University), Ahmad El
Soufi (Fran\c{c}ois Rabelais University), and Alireza
Ranjbar-Motlagh (Sharif University of Technology) and she acknowledge their support and encouragement. The author
wishes to express her thanks to Bruno Colbois and Ahmad El Soufi for suggesting the
problem and for many helpful discussions. She is also very grateful to the referee for helpful comments on the first version of the paper.
%
% first section
%
\section{Preliminaries and technical tools}
We begin by recalling some definitions.\\
\textbf{Basic definitions.} A \textit{capacitor} is a couple of Borel sets $(F,G)$ in a topological space $X$
such that $F\varsubsetneq G$.\\
We say that a metric space
$(X,d)$ satisfies the  $(\kappa,N;\rho)$-{\bf {\it covering
property}} if each ball of radius $0<r\leq \rho$ can be covered by
$N$ balls of radius $\frac{r}{\kappa}$. We sometimes call it local covering property when $\rho<\infty$.\\
For any  $x\in X$ and $0\leq r\leq R$, we define the annulus
$A(x,r,R)$ as
\begin{eqnarray*}
 A(x,r,R):=B(x,R)\setminus B(x,r)=\{y\in X : r\leq d(x,y)<R\}.
\end{eqnarray*}
Note that $A(x,0,R)=B(x,R)$. For any  annulus $A(x,r,R)$ and $\lambda\geq 1$, set $\lambda
A:=A(x,\lambda^{-1}r,\lambda R)$.   For $F\subseteq X$
and $r>0$, we denote the $r$-{\it neighborhood of} $F$  by $F^r$,
that is
\[  F^r=\{x\in X : d(x,F)\leq r\}.\]
Here, we state the key method that we use in order to obtain our results. This method was introduced in \cite{H} and was inspired by two elaborate constructions given in \cite{CM} and \cite{GNY}. It leads to construct a ``nice'' family of capacitors  which is crucial  to estimate the eigenvalues of Schr\"odinger operators and Bakry--\'Emery operators via capacities.\\

\noindent\textbf{Capacity on Riemannian manifolds.} For each capacitor $(F,G)$ in a Riemannian
manifold $(M,g)$ of dimension $m$, we define the capacity and the $m$-capacity by:
\begin{equation}
{\rm
cap}_g(F,G)=\inf_{\varphi\in\mathcal{T}}\int_M|\nabla_g\varphi|^2d\mu_g,\quad\text{and}\quad {\rm
cap}^{(m)}_{[g]}(F,G)=\inf_{\varphi\in\T}\int_M|\nabla_g\varphi|^md\mu_g,
\end{equation}
respectively, where $\T=\T(F,G)$ is the set of all functions $\varphi\in
C_0^\infty(M)$ such that ${\rm supp~}\varphi\subset G$, $0\leq\phi\leq1$ and
$\varphi\equiv1$ in a neighborhood of $F$. {If $\T(F,G)$ is empty, then ${\rm cap}_{g}(F,G) ={\rm cap}^{(m)}_{[g]}(F,G)=
+\infty$}.\\
\begin{prop}{\rm({\cite[Theorem 1.2.1]{Hthese}, see also \cite{H}}\label{decomp})}
Let $(X,d,\mu)$ be a metric measure space with {a non-atomic Borel measure $\mu$} satisfying the
{$(2,N;\rho)$}-covering property.  Then for every $n\in
\mathbb{ N}^*$,
% and $\mu(B(x,r))\leq Mr^n$ for any $r<1$
 there exists a family of capacitors
$\A=\{(F_i,G_i)\}_{i=1}^n$ with the following properties:
\begin{enumerate}[(i)]
\item$\mu(F_i)\geq\nu:=\frac{\mu(X)}{8c^2n}$,  where  $c$ is a constant depending only on $N$ ; \item the $G_i$'s are mutually disjoint ; \item
the family $\A$ is such that either
\begin{itemize}
\item[(a)]\label{chp1pa}  all the $F_i$'s are annuli and $G_i=2
F_i$, with outer radii smaller than {$\rho$}, or \item[(b)]
\label{chp1pb} all the $F_i$'s are domains in $X$ and
$G_i=F^{r_0}_i$, with $r_0 = {\frac{\rho}{1600}}$.
\end{itemize}
%\item ${\rm cap}^{(n)}_n(A_i,G_i)\leq$
\end{enumerate}
%here $\lambda$ is a constant greater than 1 and
\end{prop}
The following lemma is a consequence of the above proposition.
%
%======================    Proof of theorem   ===============================
%
\begin{lemma}\label{sumch3}
Let $(M^m,g,\mu)$ be a compact Riemannian manifold with  a non-atomic Borel measure $\mu$. Then there exist
positive constants $c(m)\in(0,1)$ and $\alpha(m)$
depending only on the dimension such that for every $k\in \N^*$ there
exists  a family $\{(F_i,G_i)\}_{i=1}^{k}$ of mutually disjoint capacitors with the following properties:
\begin{enumerate}[(I)]
 \item \label{ch3alem}${\mu}(F_i)>c(m)\frac{{\mu}(M)}{k}$ ;
  \item\label{ch3blem}  ${\rm cap}_g(F_i,G_i)\leq \frac{\mu_g(M)}{k}\left[ \frac{1}{r_0^2}\left(\frac{V([g])}{\mu_g(M)}\right)^{\frac{2}{m}}+\alpha(m)\left(\frac{k}{\mu_g(M)}\right)^{\frac{2}{m}}\right],$
\end{enumerate}
where $r_0=\frac{1}{1600}$.
\end{lemma}
\begin{proof}[Proof of Lemma \ref{sumch3}]
Take the metric measure space $(M,d_{g_0},\mu)$, where $g_0\in[g]$ with  ${\rm Ricci}_{g_0}\geq-(m-1)$ and $d_{g_0}$ is  the distance associated to the Riemannian metric $g_0$. It is easy to verify that $(M,d_{g_0},\mu)$ has the $(2,N;1)$-covering property where $N$ is a constant depending only on the dimension \cite{H}. Therefore, Proposition \ref{decomp}   implies that for every $k\in\N^*$ there is a family of
$3k$ mutually disjoint capacitors $\{(F_i,G_i)\}_{i=1}^{3k}$, satisfying  the following properties.
\begin{itemize}
\item[(a)]${\mu}(F_i)>c(m)\frac{{\mu}(M)}{k}$, where $c(m)\in(0,1)$ is a positive constant depending only on the dimension ;
\item[(b)]  all the $F_i$'s are annuli, $G_i=2 F_i$ with outer radii smaller than 1 and
${\rm cap}^{(m)}_{[g]}(F_i,2F_i)\leq Q_m$, where $Q_m$ is a constant depending only on the dimension, or 
\item[(c)]  all the
$F_i$'s are domains in $M$ and $G_i=F^{r_0}_i$ is the $r_0$-neighborhood of $F_i$, where  $r_0=\frac{1}{1600}$.
\end{itemize} 
  We  refer the reader to \cite[Proposition 3.1]{H} for more details on the proof of the part (b). Hence, the family of $\{(F_i,G_i)\}_{i=1}^{3k}$  has  the property $(I)$. \\
We now show that at least $k$ of them satisfy the property \textit{(II)}. 
We first find an upper bound for the $m$-capacity  ${\rm cap}_{[g]}^{(m)}(F_i,G_i)$. If all $F_i$'s are annuli, we already have an estimate by the part (b). In the case (c),  one can define a family of  functions $\varphi_i\in\T(F_i,G_i)$, $1\leq i\leq 3k$  so that $|\nabla_{g_0}\varphi_i|\leq\frac{1}{r_0}$. Then
$${\rm cap}_{[g]}^{(m)}(F_i,G_i)\leq\int_M|\nabla_{g_0} \varphi_i|^md\mu_{g_0}\leq\frac{1}{r_0^m}\mu_{g_0}(G_i).$$
 Since $G_1,\ldots,G_{3k}$ are
mutually disjoint, there exist  at least $2k$ of them
so that $\mu_{g_0}(G_i)\leq\mu_{g_0}(M)/k$. Similarly, there exist at least $2k$ sets (not necessarily  the
same ones) such that $\mu_g(G_i)\leq\mu_{g}(M)/k$.
Therefore, up to re-ordering, we assume that the first $k$ of them
(i.e. $G_1,\ldots,G_k$)  satisfy both of the two following inequalities
$$\mu_g(G_i)\leq\mu_g(M)/k,\quad \quad\mu_{g_0}(G_i)\leq\mu_{g_0}(M)/k.$$
Hence, in general, there exist $k$ capacitors $(F_i,G_i), 1\leq i\leq k$ with
\[{\rm cap}_{[g]}^{(m)}(F_i,G_i)\leq Q_m+\frac{1}{r_0^m}\frac{\mu_{g_0}(M)}{k}.\]
The left hand-side of the above inequality is a conformal invariant. Now, taking infimum over $g_0\in[g]$ with  ${\rm Ricci}_{g_0}\geq-(m-1)$ we get
\[{\rm cap}_{[g]}^{(m)}(F_i,G_i)\leq Q_m+\frac{1}{r_0^m}\frac{V([g])}{k}.\]
Now, for every $\varepsilon>0$, we consider plateau functions $\{f_i\}_{i=1}^k$, $f_i\in\T(F_i,G_i)$ with
$$\int_M|\nabla_{g} f_i|^md\mu_{g}\leq{\rm cap}_{[g]}^{(m)}(F_i,G_i)+\varepsilon.$$
Therefore,
\begin{align}\label{ch3welknow}
\nonumber{\rm cap}_g(F_i,G_i)\leq&\int_M|\nabla_g f_i|^2d\mu_{g}\leq\left(\int_M|\nabla_{g} f_i|^md\mu_{g}\right)^{\frac{2}{m}}\left(\int_M  1_{_{{\rm supp}f_i}}d\mu_{g}\right)^{1-\frac{2}{m}}\\
%\nonumber&=&{\left(\int_M|\nabla_{g_0} f_i|^md\mu_{g_0}\right)^{\frac{2}{m}}\left(\int_M  1_{_{{\rm supp}f_i}}d\mu_{g}\right)^{1-\frac{2}{m}}}\\
\nonumber\leq&\left({\rm cap}_{[g]}^{(m)}(F_i,G_i)+\varepsilon\right)^{\frac{2}{m}}\mu_g(G_i)^{1-\frac{2}{m}}\\
\nonumber\leq&\left(
Q_m+\frac{1}{r_0^m}\frac{V([g])}{k}+\varepsilon\right)^{\frac{2}{m}}\mu_g(G_i)^{1-\frac{2}{m}}\\
\leq&\left[Q_m^{\frac{2}{m}}+\frac{1}{r_0^2}\left(\frac{V([g])}{k}\right)^{\frac{2}{m}}+\varepsilon^{\frac{2}{m}}\right]\left(\frac{\mu_g(M)}{k}\right)^{1-\frac{2}{m}}.
 \end{align}
 where Inequality (\ref{ch3welknow}) is due to the well-know fact that
 \[(a+b)^s\leq a^s+b^s\]
 when $a,b$ are nonnegative real numbers and $0<s\leq 1$. Letting $\varepsilon$ tend to zero, we obtain the property \textit{(II)}. It completes the proof.
\end{proof}
\noindent\textbf{Capacity on Bakry--\'Emery manifolds}. In an analogous way, we define  the capacity in a Bakry--\'Emery manifold $(M,g,\phi)$.  For each capacitor $(F,G)$ in a Bakry--\'Emery manifold $(M,g,\phi)$ of dimension $m$,  the capacity and the $m$-capacity is defined as:
\begin{equation}
{\rm
cap}_\phi(F,G)=\inf_{\varphi\in\mathcal{T}}\int_M|\nabla_g\varphi|^2d\mu_\phi,\quad\text{and}\quad {\rm
cap}^{(m)}_{\phi}(F,G)=\inf_{\varphi\in\T}\int_M|\nabla_g\varphi|^md\mu_\phi,
\end{equation}
respectively, where $\T=\T(F,G)$ is the set of all functions $\varphi\in
C_0^\infty(M)$ such that ${\rm supp~}\varphi\subset G$, $0\leq\phi\leq1$ and
$\varphi\equiv1$ in a neighborhood of $F$. {If $\T(F,G)$ is empty, then ${\rm cap}_{\phi}(F,G) ={\rm cap}^{(m)}_{\phi}(F,G)=
+\infty$}.\\ We shall prove a similar lemma as Lemma \ref{sumch3}. We start by showing that every compact Bakry--\'Emery manifold satisfies the assumptions of Proposition \ref{decomp}.
%
%======          wei wylie comparison thm    =============================
Thanks to volume comparison theorem proved by Wei and Wylie \cite{WW} for  Bakry--\'Emery manifolds, one can show that Bakry--\'Emery manifolds have local covering property (see Lemma \ref{3covch} below).
\begin{thm}[\textit{Volume comparison theorem}\cite{WW}]Let
$(M,g,\phi)$ be {a compact Bakry}--\'Emery manifold with
$Ricci_\phi \geq\alpha(m-1)$. If  $\partial_r\phi\geq-\sigma$, with respect to geodesic polar
coordinates centered at $x$,  then for every $0<r\leq R$ we have (assume $R\leq\pi/2\sqrt{\alpha}$
 if $\alpha>0$)
\begin{equation}\label{wwcom}
\frac{\mu_\phi(B(x,R))}{\mu_\phi(B(x,r))}\leq e^{\sigma
R}\frac{v(m,R,\alpha)}{v(m,r,\alpha)},
\end{equation}
and in particular, letting $r$ tend to zero yields
\begin{equation}\label{combak3}\mu_\phi(B(x,R))\leq e^{\sigma
R}v(m,R,\alpha),
\end{equation}
where $v(m,r,\alpha)$ is the volume of a ball of radius $r$ in the simply connected  space form of constant sectional curvature $\alpha$.
\end{thm}
\begin{lemma}\label{3covch}
Let $(M,g,\phi)$ be a {compact} Bakry--\'Emery
manifold with $Ricci_\phi\geq-\kappa^2(m-1)$ and
$|\nabla_g\phi|\leq\sigma$ for some {$\kappa\geq0$ and $\sigma\geq0$}. There exist constants  $N(m)\in\N^*$ and
$\xi=\xi(\sigma,\kappa)>0$  such that $(M,g,\phi)$ satisfies  the $(2,N;\xi)$-covering property. Moreover, there exists a positive  constant $C(m)$ such that
for every $0\leq r<R\leq\xi$ and $x\in M$,  the annulus   $A=A(x,r,R)$ satisfies  ${\rm cap}^{(m)}_\phi(A,2A))\leq C(m)$.
\end{lemma}
\begin{proof}Take
$\xi=min\{\frac{1}{\sigma},\frac{1}{\kappa}\}$ {(take $\xi=\infty$ if $\sigma=\kappa=0$ )}.  We first show that  $(M,\mu_\phi)$ has the doubling property for $r<4\xi$,  i.e.
$$\mu_\phi(B(x,r))\leq c\mu_\phi(B(x,r/2)),\quad 0<r<4\xi,$$
for some positive constant $c$. From this, it is easy to deduce that
  $(M,\mu_\phi)$  has
the $(2,N;\xi)$-covering property for example with $N=c^4$.  To prove the doubling property, according to Inequality (\ref{wwcom}) we have
\begin{eqnarray*}
\frac{\mu_\phi(B(x,r))}{\mu_\phi(B(x,r/2))}\leq e^{\sigma
r}\frac{v(m,r,-\kappa^2)}{v(m,r/2,-\kappa^2)}
=
e^{\sigma r}\frac{v(m,\kappa r,-1)}{v(m,\kappa r/2,-1)}.
\end{eqnarray*}
Take $\tilde{r}:=\kappa r$ and $\tilde{R}:=\kappa R$. Hence, for every $0<{r}<4\xi=4\min\{\frac{1}{\sigma},\frac{1}{\kappa}\}$, we get
\begin{eqnarray*}
e^{\sigma r}\frac{v(m,\kappa r,-1)}{v(m,\kappa r/2,-1)}&\leq&e^4\frac{v(m,\tilde{r},-1)}{v(m,\tilde{r}/2,-1)} ;\quad0<\tilde{r}<4,\\
&\leq&\sup_{\tilde{r}\in(0,4)}e^4\frac{v(m,\tilde{r},-1)}{v(m,\tilde{r}/2,-1)}=:c(m).
\end{eqnarray*}
Thus,\[\frac{\mu_\phi(B(x,r))}{\mu_\phi(B(x,r/2))}\leq c(m),\quad \text{for every}~0<{r}<\xi.\]
 Therefore, $(M,g,\phi)$ has $(2,N;\xi)$-covering
property where $N=c^4(m)$.\\
To estimate the capacity of an annulus, we now follow the same argument as in \cite[page 3430]{H}. Let $A=A(x,r,R)$ and let $f\in\mathcal{T}(A,2A)$ be \begin{equation}\label{plateauf2}
f(y)=\left\{
\begin{array}{clll}
    1 & {\rm if } & y\in A(x,r,R) \\
    \frac{2d_{g_0}(y,B(x,r/2))}{r} & {\rm if } & y\in A(x,r/2,r)~{\rm and}~ r\neq0\\
    1-\frac{d_{g_0}(y,B(x,R))}{R} & {\rm if } & y\in A(x,R,2R)\\
    0  & {\rm if } & y\in M\setminus A(x,r/2,2R)\
\end{array}\right..
\end{equation} We have
\[|\nabla_{g_0} f|\leq\frac{2}{r},\quad {\rm on}~B(x,r)\setminus B(x,r/2),\]\[\quad |\nabla_{g_0} f|\leq\frac{1}{R},\quad {\rm~on}~ B(x,2R)\setminus B(x,R).\]
Therefore,
%\[{\rm cap}^{(m)}_\phi(A,2A)\leq C(m).\]
%Indeed (refer to page \pageref{tabehalghe}),
\begin{align*}
{\rm cap}^{(m)}_\phi(A,2A)&\leq\int_M|\nabla_g
f|^md\mu_\phi\leq\big(\frac{2}{r}\big)^m\mu_{\phi}(A(x,r/2,r))+\big(\frac{1}{R}\big)^m\mu_{\phi}(A(x,R,2R))\\
&\leq
\big(\frac{2}{r}\big)^m\mu_{\phi}(B(x,r))+\big(\frac{1}{R}\big)^m\mu_{\phi}(B(x,2R)).
\end{align*}
Having Inequality  (\ref{combak3}), one gets
\begin{align*}
{\rm cap}^{(m)}_\phi(A,2A)\leq&\left(\frac{2}{r}\right)^me^{\sigma
r}v(m,r,-\kappa^2)+\left(\frac{1}{R}\right)^me^{2\sigma R
}v(m,2R,-\kappa^2)\\
=&\left(\frac{2}{\kappa r}\right)^me^{\sigma r}v(m,\kappa
r,-1)+\left(\frac{1}{\kappa R}\right)^me^{2\sigma R
}v(m,2\kappa R,-1).
\end{align*}
Take $\tilde{r}:=\kappa r$ and $\tilde{R}:=\kappa R$. Hence, for every $0<{r}<R\leq2\xi=2\min\{\frac{1}{\sigma},\frac{1}{\kappa}\}$, we get
\begin{align}\label{injadef}
\nonumber {\rm cap}^{(m)}_\phi(A,2A)&\leq\left(\frac{2}{\tilde{r}}\right)^me^2v(m,\tilde{r},-1)+\left(\frac{1}{\tilde{R}}\right)^me^{4
}v(m,2\tilde{R},-1)
\\
 \nonumber  &\leq\sup_{\tilde{r},\tilde{R}\in(0,2)}\left[\left(\frac{2}{\tilde{r}}\right)^me^2 v(m,\tilde{r},-1)+\left(\frac{1}{\tilde{R}}\right)^me^{4
}v(m,2\tilde{R},-1)\right]\\
&=:C(m).
\end{align}
This completes the proof.
\end{proof}
%Finally, we get the following lemma similar to Lemma \ref{sumch3}.
{\begin{lemma}\label{baklemch3}
Let $(M^m,g,\phi)$ be {a compact Bakry}--\'Emery manifold with $|\nabla_g\phi|\leq\sigma$ for some $\sigma\geq0$. Then there exist
positive constants $c(m)\in(0,1)$ and $\alpha(m)$
depending only on the dimension such that for every $k\in \N^*$ there
exists  a family $\{(F_i,G_i)\}_{i=1}^{k}$ of capacitors with the following properties:
\begin{enumerate}[(I)]
 \item ${\mu_\phi}(F_i)>c(m)\frac{{\mu_\phi}(M)}{k}$,
  \item ${\rm cap}_\phi(F_i,G_i)\leq \frac{\mu_\phi(M)}{k}\left[ \frac{1}{r_0^2}\left(\frac{V_\phi([g])}{\mu_\phi(M)}\right)^{\frac{2}{m}}+\alpha(m)\left(\frac{k}{\mu_\phi(M)}\right)^{\frac{2}{m}}\right]$, 
\end{enumerate}
where $\frac{1}{r_0}=1600\max\{\sigma,1\}$.
\end{lemma}}

\begin{proof}{
We consider  the Bakry--\'Emery manifold $(M,g,\phi)$ as the  metric measure  space $(M,d_{g_0},\mu_\phi)$ where $g_0\in[g]$ with ${\rm Ricci}_\phi(M,g_0)\geq-(m-1)$ and $\mu_\phi$ is the weighted measure  with respect to the metric $g$. According to Lemma \ref{3covch}, this space has the $(2,N,\xi)$-covering property with $\xi=\min\{\frac{1}{\sigma},1\}$. Having Proposition \ref{decomp} and Lemma  \ref{3covch}, and following steps analogous to those in Lemma \ref{sumch3}, implies that for every $k\in\N^*$, there exists  a family of $k$ mutually disjoint capacitors $\{F_i,G_i\}$ satisfying the following properties.
\begin{itemize}
\item[(a)]
$\mu_\phi(F_i)\geq c(m)\frac{\mu_\phi(M)}{k},$
where $c(m)\in(0,1)$ is a positive constant depending only on the dimension, and  $\mu_\phi(G_i)\leq \frac{\mu_\phi(M)}{k}$.
\item[(b)]  all the $F_i$'s are annuli, $G_i=2 F_i$ with outer radii smaller than $\xi$  and ${\rm cap}_{\phi}^{(m)}(F_i,G_i)\leq C(m)$, where $C(m)$ is a constant defined in \eqref{injadef}.
or
\item[(c)]  all the
$F_i$'s are domains in $M$, $G_i=F^{r_0}_i$ is the $r_0$-neighborhood of $F_i$  and ${\rm cap}_{\phi}^{(m)}(F_i,G_i)\leq\frac{1}{r_0^2}\frac{V_{\phi}([g])}{k}$, with $r_0=\frac{\xi}{1600}$ 
\end{itemize}  
Hence, ${\rm cap}_{\phi}^{(m)}(F_i,G_i)\leq C(m)+\frac{1}{r_0^2}\frac{V_{\phi}([g])}{k}$.
Now,  for every $\varepsilon>0$, we consider a family of functions $\{f_i\}_{i=1}^k$, $f_i\in\T(F_i,G_i)$ such that
$$\int_M|\nabla_{g} f_i|^me^{-\phi}d\mu_{g}\leq{\rm cap}_{\phi}^{(m)}(F_i,G_i)+\varepsilon.$$
We repeat the same argument as before.  
\begin{align*}
\nonumber{\rm cap}_\phi(F_i,G_i)&\leq\int_M|\nabla_g f_i|^2e^{-\phi}d\mu_{g}\\
&\leq\left(\int_M|\nabla_{g} f_i|^me^{-\phi}d\mu_{g}\right)^{\frac{2}{m}}\left(\int_M  1_{_{{\rm supp}f_i}}e^{-\phi}d\mu_{g}\right)^{1-\frac{2}{m}}\\
&\leq\left[C(m)^{\frac{2}{m}}+\frac{1}{r_0^2}\left(\frac{V_{\phi}([g])}{k}\right)^{\frac{2}{m}}+\varepsilon^{\frac{2}{m}}\right]\left(\frac{\mu_\phi(M)}{k}\right)^{1-\frac{2}{m}}.
 \end{align*}
 Having $\frac{1}{r_0}=\frac{1600}{\xi}=1600\max\{\sigma,1\}$ and letting $\varepsilon$ tend to zero, we obtain the property (II). It completes the proof.}
\end{proof}
%
%    Schrodinger operators
%
%
\section{Eigenvalues of Schr\"odinger operators}\label{sec2p}
In this section, we prove Theorems \ref{thm} and \ref{pot}.
The idea of the proof is to construct a suitable family of test functions to be used in the variational characterization of the eigenvalues. Due to the min-max Theorem, we have
the following variational characterization for the eigenvalues of
the Schr\"odinger operator $L=\Delta_g+q$:
\begin{equation*}
\lambda_k{(\Delta_g+q)}=\min_{V_k}\max_{0\neq f\in
V_k}\frac{\int_M|\nabla_g
f|^2d\mu_g+\int_Mf^2qd\mu_g}{\int_Mf^2d\mu_g},
\end{equation*}
where $V_k$ is a $k$-dimensional linear subspace of $H^1(M)$ and
$\mu_g$ is the Riemannian measure corresponding to the metric $g$.\\
 According to this variational formula, for every family
$\{f_i\}_{1=1}^k$ of disjointly supported  test functions one has
\begin{eqnarray}\label{11}
 \lambda_k{(\Delta_g+q)}\leq \max_{i\in\{1,\ldots,k\}}\frac{\int_M|\nabla_g
f_i|^2d\mu_g+\int_Mf_i^2qd\mu_g}{\int_Mf_i^2d\mu_g}.
\end{eqnarray}
The potential $q\in C^0(M)$  is a signed function ({notice that we can assume $q\in L^1(M)$ as well}). We define a signed measure $\sigma$ associated to the potential $q$ by
   \[\sigma(A)=\int_Aqd\mu_g,\quad\text{for every measurable subset}~A~\text{of}~X.\]
 For any signed measure $\nu$ we write $\nu=\nu^+-\nu^-$, where $\nu^+$ and $\nu^-$ are the positive and negative parts of $\nu$, respectively. For any signed measure $\nu$ and $0\leq\delta\leq1$ we define a new signed measure $\nu_\delta$ as $\nu_\delta:=\delta\nu^+-\nu^-$.\\
Let $\mu$ and $\nu$ be two signed measures on $M$. Then, according to \cite[Lemma 4.3]{GNY}, the following inequality is satisfied.
\begin{equation}\label{lemch3measur}(\mu+\nu)_\delta\geq\mu_\delta+\nu_\delta.\end{equation}
\begin{proof}[{ Proof of Theorem \ref{thm}}] For a  real number $\lambda\in\R$  define $\mu_\lambda:=(\lambda\mu_g-\sigma)^+$ as a non-atomic Borel
measure on $M$. We apply Lemma \ref{sumch3} to $(M,g,\mu_\lambda)$. Thus, for every $k\in\N^*$ and every $\lambda\in\R$, there exists a family  $\{(F_i,G_i)\}_{i=1}^{2k}$ of $2k$ capacitors satisfying the properties (\ref{ch3alem}) and (\ref{ch3blem}) of Lemma \ref{sumch3}.\\
From now on, we take $\lambda:=\lambda_k=\lambda_k{(L)}$. The property (\ref{ch3alem}) yields
$$(\lambda_k\mu_g-\sigma)^+(F_i)\geq c(m)\frac{(\lambda_k\mu_g-\sigma)^+(M)}{2k}.$$
% Suppose that $(\lambda\mu_g-\sigma)^-\neq0$.
The measure $(\lambda_k\mu_g-\sigma)^-$  is also a non-atomic. Since $G_i$'s are
mutually disjoint, up to reordering, 
%there exist  at least $k$ of them with measure not greater than $\frac{(\lambda_k{(L)}\mu_g-\sigma)^-(M)}{k}$.   Up to reordering, assume that 
the first $k$ of them satisfy
$$(\lambda_k\mu_g-\sigma)^-(G_i)\leq\frac{(\lambda_k\mu_g-\sigma)^-(M)}{k},\quad i\in\{1,\ldots,k\}.$$
Therefore 
\begin{align}\label{5}
\nonumber(\lambda_k\mu_g-\sigma)^-(G_i)-(\lambda_k\mu_g-\sigma)^+(F_i)\leq\frac{(\lambda_k\mu_g-\sigma)^-(M)}{k}\\
-c(m)\frac{(\lambda_k\mu_g-\sigma)^+(M)}{2k}.
\end{align}
For every $\epsilon>0$ and every $1\leq i\leq k$, we choose
$f_i\in\mathcal{T}(F_i,G_i)$ such that:
\begin{equation}\label{2}\int_M|\nabla_g f_i|^2d\mu_g\leq {\rm cap}_g(F_i,G_i)+\epsilon.
\end{equation}
Inequality (\ref{11}) implies that there exists $i\in\{1,\cdots,k\}$ so
that
\begin{eqnarray*}
\lambda_k{\int_Mf_i^2d\mu_g}\leq\int_M|\nabla_g
f_i|^2d\mu_g+\int_Mf_i^2qd\mu_g.
\end{eqnarray*}
Hence, having Lemma \ref{sumch3} and Inequality (\ref{5}) we get
\begin{eqnarray}
\nonumber0&\leq&\int_M|\nabla_g
f_i|^2d\mu_g-\int_Mf_i^2(\lambda_k-q)d\mu_g\\
\nonumber&\leq& {\rm cap}_g(F_i,G_i)+\epsilon-\int_Mf_i^2(\lambda_k-q)d\mu_g\\
\nonumber&\leq&\frac{\mu_g(M)}{{2}k}\left[ \frac{1}{r_0^2}\left(\frac{V([g])}{\mu_g(M)}\right)^{\frac{2}{m}}+\alpha(m)\left(\frac{{2}k}{\mu_g(M)}\right)^{\frac{2}{m}}\right]+\epsilon\\
\nonumber&&+\int_Mf_i^2(\lambda_k-q)^-d\mu_g-\int_Mf_i^2(\lambda_k-q)^+d\mu_g\\
\nonumber&\leq&\frac{\mu_g(M)}{{2}k}\left[ \frac{1}{r_0^2}\left(\frac{V([g])}{\mu_g(M)}\right)^{\frac{2}{m}}+\alpha(m)\left(\frac{{2}k}{\mu_g(M)}\right)^{\frac{2}{m}}\right]+\epsilon\\
\label{inqch3measur}&&+\frac{(\lambda_k\mu_g-\sigma)^-(M)}{k}-
c(m)\frac{(\lambda_k\mu_g-\sigma)^+(M)}{2k}.
\end{eqnarray}
We now estimate the last two terms of the above inequality considering two alternatives:

\textbf{Case 1.~} If  $\lambda_k=\lambda_k{(L)}$ is positive, then applying Inequality \eqref{lemch3measur} for the {measure $\lambda_k\mu_g$ and signed measure $-\sigma$ with $\delta=\frac{c(m)}{2}$}, we get
\begin{align}\label{deltainqch3}\nonumber\frac{c(m)}{2}(\lambda_k\mu_g-\sigma)^+(M)-(\lambda_k\mu_g-\sigma)^-(M)&\geq\frac{c(m)}{2}\sigma^-(M)-\sigma^+(M)\\
&+\frac{c(m)}{2}\lambda_k\mu_g(M).\end{align}
Replacing {\eqref{deltainqch3}} in (\ref{inqch3measur}), {and   letting $\epsilon$ tend to zero} gives the following
\begin{eqnarray}\label{am}
\lambda_k\leq
\frac{\frac{2}{c(m)}\sigma^+(M)-\sigma^-(M)}{\mu_g(M)}
+\frac{1}{c(m)r_0^2}\left(\frac{V([g])}{\mu_g(M)}\right)^{\frac{2}{m}}+
\frac{\alpha(m)}{c(m)}\left(\frac{2k}{\mu_g(M)}\right)^{\frac{2}{m}}.
 \end{eqnarray}

\textbf{Case 2.~} If $\lambda_k=\lambda_k{(L)}$ is non-positive, then  applying Inequality \eqref{lemch3measur} for the {signed measures $\lambda_k\mu_g$ and  $-\sigma$ with $\delta=\frac{c(m)}{2}$}, implies
\begin{align}\label{deltainq2ch3}\nonumber\frac{c(m)}{2}(\lambda_k\mu_g-\sigma)^+(M)-(\lambda_k\mu_g-\sigma)^-(M)\geq&~
\frac{c(m)}{2}\sigma^-(M)-\sigma^+(M)\\
&+\lambda_k\mu_g(M).\end{align}
Replacing {\eqref{deltainq2ch3}} in (\ref{inqch3measur}) {and letting} $\epsilon$ go to zero gives the following
 \begin{align}\label{be}
\lambda_k\leq
\frac{\sigma^+(M)-\frac{c(m)}{2}\sigma^-(M)}{\mu_g(M)}
+\frac{1}{{2}r_0^2}\left(\frac{V([g])}{\mu_g(M)}\right)^{\frac{2}{m}}+
\frac{\alpha(m)}{2}\left(\frac{2k}{\mu_g(M)}\right)^{\frac{2}{m}}.
 \end{align}
Therefore, $\lambda_k{(L)}$ is smaller than the sum of the right-hand sides of Inequalities (\ref{am}) and (\ref{be}). We finally obtain
Inequality (\ref{222}) with, for example, $\alpha_m=\frac{c(m)}{4}$.
\end{proof}
%========================== proof of theorem 2 =======================
\begin{proof}[{ Proof of Theorem \ref{pot} }] We partly follow the spirit of the proof of \cite[Theorem 5.15]{GNY}.
 Take the measure metric space $(M,g,\mu_g)$. By
Lemma \ref{sumch3}, for every  $k\in N^*$ there is a family of $2k$ disjoint capacitors
$\{(F_i,G_i)\}_{i=1}^{2k}$ that satisfies the properties (\ref{ch3alem}) and
(\ref{ch3blem}). For every $\varepsilon>0$, let  $\{f_i\}_{i=1}^{2k}$ be a family  of test functions with  $2f_i\in\mathcal{T}(F_i,G_i)$ and $4\int_M|\nabla_g f_i|^2d\mu_g\leq{\rm cap}_g(F_i,G_i)+\varepsilon$. We  claim that this family satisfies the following property:\\
\begin{equation}\label{ch3inyau}
\sum_{i=1}^{2k}\int_Mf_i^2qd\mu_g\leq\sum_{i=1}^{2k}\int_M|\nabla_g
f_i|^2d\mu_g+\int_M qd\mu_g.
\end{equation} If  we
have Inequality (\ref{ch3inyau}) then
\begin{align*}
\sum_{i=1}^{2k}\int_M\left(|\nabla_g
f_i|^2+f_i^2q\right)d\mu_g&\leq2\sum_{i=1}^{2k}\int_M|\nabla_g
f_i|^2d\mu_g+\int_M qd\mu_g\\
&\leq {k}\max_i{\rm cap}_g(F_i,G_i)+k\varepsilon+\int_M qd\mu_g.
\end{align*}
By the assumption $\int_M\left(|\nabla_g f_i|^2+f_i^2q\right)d\mu_g$ is positive for each $1\leq i\leq2k$.
Therefore, at least $k$ of them satisfy the following inequality (up to reordering we assume that the first  $k$ of them satisfy the inequality):
\begin{equation}\label{mm}\int_M(|\nabla_g
f_i|^2+f_i^2q)d\mu_g\leq\max_i{\rm cap}_g(F_i,G_i)+\varepsilon+\frac{\int_M qd\mu_g}{k}.
\end{equation}
Inequality (\ref{mm}) together with the bounds of ${\rm cap}_g(F_i,G_i)$  and $\mu_g(F_i)$ given in
Lemma \ref{sumch3}, properties  (\ref{ch3alem}) and
(\ref{ch3blem}) lead to
{\begin{align*}
\lambda_k{(L)}&\leq\max_i\frac{\int_M|\nabla_g
f_i|^2d\mu_g+\int_Mf_i^2qd\mu_g}{{\int_Mf_i^2d\mu_g}}\leq\frac{\max_i{\rm cap}_g(F_i,G_i)+\varepsilon+\frac{1}{k}\int_M qd\mu_g}{\mu_g(F_i)}\\
&\leq\frac{1}{c(m)r_0^2}\left(\frac{V([g])}{\mu_g(M)}\right)^{\frac{2}{m}}+\alpha(m)\left(\frac{2k}{\mu_g(M)}\right)^{\frac{2}{m}}+\frac{2k\varepsilon}{c(m)\mu_g(M)}+\frac{2\int_Mqd\mu_g}{c(m)\mu_g(M)}.
%&\hspace{-.8cm}= A_m\frac{\int_M qd\mu_g}{\mu_g(M)}+B_m\left(\frac{V([g])}{\mu_g(M)}\right)^{\frac{2}{m}}+
%C_m\left(\frac{k}{\mu_g(M)}\right)^{\frac{2}{m}}+D_m\frac{k\varepsilon}{\mu_g(M)},
\end{align*}}
Hence, we get the desired inequality as  $\varepsilon$ tends to zero.
%This completes the proof of Theorem \ref{pot}.\\
It remains to
prove Inequality (\ref{ch3inyau}) which is proved in \cite[Section 5]{GNY}; however, for the reader's convenience
 we repeat the proof. We define the function $h$ by the following identity \begin{equation}\label{idch3h}\sum_{i=1}^{2k}f_i^2+h^2=1.\end{equation}
Since $f_1,\ldots,f_{2k}$ are disjointly  supported and $0\leq f_i\leq\frac{1}{2}$, hence, $h\geq\frac{1}{2}$.
We now estimate the left-hand side of Inequality (\ref{ch3inyau}).
\begin{equation}
\label{idch3h1}\int_M\left(\sum_{i=1}^{2k}f_i^2+h^2-h^2\right)qd\mu_g=\int_Mqd\mu_g-\int_Mh^2qd\mu_g\leq\int_Mqd\mu_g+\int_M|\nabla h|^2d\mu_g,
\end{equation}
where the last inequality comes from the fact that the Schr\"odinger operator $L$ is positive.
Identity (\ref{idch3h}) implies
\[-2h\nabla_g h=-\nabla_g h^2=\sum_{i=1}^{2k}\nabla_g f_i^2=2\sum_{i=1}^{2k}f_i\nabla_g f_i.\]
Therefore,
\begin{equation}
\label{idch3h2}|\nabla_g h|^2\leq|2h\nabla_g h|^2=\sum_{i=1}^{2k}|\nabla_g f_i^2|^2=4\sum_{i=1}^{2k}|f_i\nabla_g f_i|^2\leq\sum_{i=1}^{2k}|\nabla_g f_i|^2.
\end{equation}
Combining Inequalities (\ref{idch3h1}) and (\ref{idch3h2}) we get  Inequality (\ref{ch3inyau}).
\end{proof}
%\chapter{Eigenvalues of perturbed Laplace operators}
%\selectlanguage{french}
%
%
%===========      Eigenvalues of the Bakry--emery laplacian   =================
%
\section{Eigenvalues of  Bakry--\'Emery Laplace operators}
 In this section we consider
eigenvalues of the Bakry--\'Emery Laplace operator $\Delta_\phi$ on a
Bakry--\'Emery manifold $(M,g,\phi)$, where $M$ is a
compact $m$-dimensional Riemannian manifold and
$\phi\in C^2(M)$. We denote the weighted measure on $M$ by
$\mu_\phi$\label{indexbak} with
\[\mu_\phi(A)=\int_Ae^{-\phi}d\mu_g,\quad \text{for every Borel subset}~A~\text{of}~ M.\]
\begin{proof}[Proof of Theorem \ref{bakint}] As we mentioned in the introduction, one can see that $\Delta_\phi=\Delta_g+\nabla_g\phi\cdot\nabla_g$ is unitarily equivalent to the positive Schr\"odinger operator   $L=\Delta_g+\frac{1}{2}\Delta_g\phi+\frac{1}{4}|\nabla_g\phi|^2$. Therefore,   Theorem \ref{pot} yields \begin{eqnarray*}\lambda_k{(\Delta_\phi)}&\leq& A_m\frac{1}{\mu_g(M)}\int_M\left(\frac{1}{2}\Delta_g\phi+\frac{1}{4}|\nabla_g\phi|^2\right)d\mu_g\\
&&+B_m\left(\frac{V([g])}{\mu_g(M)}\right)^{\frac{2}{m}}+
C_m\left(\frac{k}{\mu_g(M)}\right)^{\frac{2}{m}}.\end{eqnarray*}
Stokes theorem implies that $\int_M\Delta_g\phi d\mu_g=0$. This gives the result.
\end{proof}
For the proof of Theorem \ref{bakmain}, we  use the characteristic variational formula
   for the Bakry--\'Emery Laplacian
(see for example \cite[Proposition 1]{LR} and \cite[Proposition 4]{Ro}).
\begin{eqnarray}\label{varbakch3}
\lambda_k{(\Delta_\phi)}=\inf_{ V_k}\sup_{f\in V_k}\frac{\int_M|\nabla_g
f|^2e^{-\phi} d\mu_g}{\int_Mf^2e^{-\phi} d\mu_g},
\end{eqnarray}
where $V_k$ is a $k$-dimensional linear subspace of $H^1(M,\mu_\phi)$.
{\begin{proof}[Proof of Theorem \ref{bakmain}]
According to Lemma \ref{baklemch3} for $k\in\N^*$ we have a family of $k$ capacitors satisfying  properties (I) and (II).
 %\begin{equation}
 %\label{articl2} {\mu_\phi}(F_i)>c(m)\frac{{\mu_\phi}(M)}{k},\end{equation}
 %\begin{equation}
  %\label{articl3} {\rm cap}_\phi(F_i,G_i)\leq \frac{\mu_\phi(M)}{k}\left[ \max\{\sigma^2,1\}\left(\frac{V_\phi([g])}{\mu_\phi(M)}\right)^\frac{2}{m}+\alpha(m)\left(\frac{k}{\mu_\phi(M)}\right)^\frac{2}{m}\right].
%\end{equation}
 For every $\varepsilon>0$, take $f_i\in\T(F_i,G_i)$, $1\leq i\leq k$,
so that
\[\int_M|\nabla_g f_i|^2e^{-\phi}d\mu_g\leq{\rm cap}_\phi(F_i,G_i)+\varepsilon.\]
Hence, the characteristic variational formula (\ref{varbakch3}) gives
\begin{eqnarray*}
\lambda_k{(\Delta_\phi)}\leq \max_i\frac{\int_M|\nabla_g
f_i|^2e^{-\phi} d\mu_g}{\int_Mf_i^2e^{-\phi} d\mu_g}\leq\max_i\frac{{\rm cap}_\phi(F_i,G_i)+\varepsilon}{\mu_\phi(F_i)}.
\end{eqnarray*}
 Having the properties (I) and (II), we get
\begin{align*}
\lambda_k{(\Delta_\phi)}\leq A(m)\max\{\sigma^2,1\}\left(\frac{V_\phi([g])}{\mu_\phi(M)}\right)^{\frac{2}{m}}+B(m)\left(\frac{k}{\mu_\phi(M)}\right)^{\frac{2}{m}}+\frac{k\varepsilon}{c(m)\mu_\phi(M)}.
\end{align*}
Letting $\varepsilon$ go to zero, we get the desired inequality.
\end{proof}}
\appendix
\section{Buser type upper bound on Bakry-\'Emery manifolds} \label{dovomii}
%\addcontentsline{toc}{section}{Annexes}
{Here, we present a direct and simple proof of a weaker version of Corollary \ref{bakbuser}. This idea of proof was used by Buser \cite[Satz 7]{Bu2}, Cheng \cite{Ch}, Li and Yau \cite{LY} in the case of the Laplace--Beltrami operator. It is  based on constructing a family of balls as capacitors which shall be the support of test functions. We can successfully apply this idea in the case of the Bakry--\'Emery Laplace operator.}
\begin{thm}[Buser type upper bound]\label{bakbuserapp}Let $(M,g,\phi)$ be {a compact Bakry}--\'Emery manifold with
${\rm Ricci}_\phi(M)>-\kappa^2(m-1)$ and
$|\nabla_g\phi|\leq\sigma$ for some $\kappa\geq0$ and $\sigma\geq0$. There are positive constants $A(m)$
and $B(m)$ such that for every $k\in \N^*$
\begin{equation*}
\lambda_k{(\Delta_\phi)}\leq A(m)\max\{\sigma,\kappa\}^2+
B(m)\left(\frac{k}{\mu_\phi(M)}\right)^{\frac{2}{m}}.
\end{equation*}
\end{thm}
{To see that the above theorem is weaker than Corollary \ref{bakbuser}, consider the case where ${\rm Ricci}_\phi(M,g)$ is nonnegative. Indeed, the upper bound in Theorem \ref{bakbuserapp} still depends on $\sigma$ while Corollary \ref{bakbuser} provides an upper bound which depends only on the dimension. }
\begin{proof} Since ${\rm Ricci}_\phi(M)>-\kappa^2(m-1)$ and
$|\nabla_g\phi|\leq\sigma$, the comparison theorem gives us the following inequalities for every $0<r\leq\xi=\min\{\frac{1}{\sigma},\frac{1}{\kappa}\}$ (with $\xi=\infty$ if $\sigma=\kappa=0$):
\begin{eqnarray*}
\frac{\mu_\phi(B(x,r))}{\mu_\phi(B(x,r/2))}\leq e^{\sigma
r}\frac{v(m,r,-\kappa^2)}{v(m,r/2,-\kappa^2)}\leq\sup_{r\in(0,\xi)}e^{\sigma
r}\frac{v(m,r,-\kappa^2)}{v(m,r/2,-\kappa^2)}=:c_1(m),
\end{eqnarray*}
and
\begin{eqnarray}\label{comappa}
\nonumber\mu_\phi(B(x,r))\leq e^{\sigma
r}v(m,r,-\kappa^2)\leq\sup_{s\in(0,\xi)}e^{\sigma
s}v(m,s,-\kappa^2)r^m
=:c_2(m)r^m.
\end{eqnarray}
Given $k\in\N^*$ let $\rho(k)$ be the positive number defined by
\[\rho(k)=\sup\{r : \exists p_i,\ldots,p_k\in M~\text{with}~d_g(p_i,p_j)> r, \forall i\neq j\}.\]
We consider two alternatives:

\textbf{Case 1.~} Let $\rho(k)\geq\xi$. For every $r<\xi$, there are $k$ points
$p_1,\ldots,p_k$ with $B(p_i,r/2)\cap B(p_j,r/2)=\emptyset$, $\forall i\neq j$.  For each $i\in\{1,\ldots,k\},$ we consider a  plateau functions
$f_i\in\T(B(p_i,r/4),B(p_i,r/2))$, $1\leq i\leq k$,  defined as in (\ref{plateauf2}). Then, for every $1\leq i\leq k$ and every $r<\xi$
\[\frac{\int_M|\nabla_g f_i|^2e^{-\phi}d\mu_g}{\int_Mf_i^2e^{-\phi}d\mu_g}\leq\frac{16}{r^2}\frac{\mu_\phi(B(p_i,r/2))}{\mu_\phi(B(p_i,r/4))}\leq c_1(m)\frac{16}{r^2}.\] Therefore, letting $r$ tend to $\xi$, one has
\[\frac{\int_M|\nabla_g f_i|^2e^{-\phi}d\mu_g}{\int_Mf_i^2e^{-\phi}d\mu_g}\leq c_1(m)\frac{16}{\xi^2}\leq A(m)\max\{\sigma,\kappa\}^2.\]
\textbf{Case 2.~} Let $\rho(k)< \xi$.  Take $r<\rho(k)$ very close to
$\rho(k)$. As in Case 1, there are $k$ points $p_1,\ldots,p_k$
with $B(p_i,r/2)\cap B(p_j,r/2)=\emptyset$, $\forall i\neq j$. Repeating the same argument we get for every $1\leq i\leq k$
\[\frac{\int_M|\nabla_g f_i|^2e^{-\phi}d\mu_g}{\int_Mf_i^2e^{-\phi}d\mu_g}\leq c_1(m)\frac{16}{r^2}.\]
%The above inequality satisfies for every $r<\rho(k)$. Hence, we
%also have for $\rho(k)$.
 Therefore, for every $1\leq i\leq k$
\[\frac{\int_M|\nabla_g f_i|^2e^{-\phi}d\mu_g}{\int_Mf_i^2e^{-\phi}d\mu_g}\leq c_1(m)\frac{16}{\rho(k)^2}.\]
We now estimate $\rho(k)$.
 Let $\rho(k)<s<\xi$ and  $n$ be the maximal
number of points $q_1,\ldots,q_n\in M$ so that $d(q_i,q_j)>s$, $\forall i\neq j$. Of
course $n\leq k$ and because of the maximality of $n$,  the balls
$\{B(q_i,s)\}_{i=1}^n$ cover $M$. Hence, according to Inequality (\ref{comappa})
$$\mu_\phi(M)\leq\sum_{i=1}^n\mu_\phi(B(q_i,s))\leq  n c_2(m) s^m\leq k c_2(m) s^m.$$
Thus, letting $s$ tend to $\rho(k)$ we get
%\[r\geq\frac{1}{3\beta^{1/p}}\left(\frac{\mu(X)}{n}\right)^{1/p}\]
\[\frac{1}{\rho(k)^2}\leq c_2(m)^{\frac{2}{m}}\left(\frac{k}{\mu_\phi(M)}\right)^{\frac{2}{m}}.\]
Therefore,
\[\frac{\int_M|\nabla_g f_i|^2e^{-\phi}d\mu_g}{\int_Mf_i^2e^{-\phi}d\mu_g}\leq16c_1(m)c_2(m)^{\frac{2}{m}}\left(\frac{k}{\mu_\phi(M)}\right)^{\frac{2}{m}}.\]
In conclusion, we obtain
\[\lambda_k{(\Delta_\phi)}\leq\max_i\frac{\int_M|\nabla_g f_i|^2e^{-\phi}d\mu_g}{\int_Mf_i^2e^{-\phi}d\mu_g}\leq A(m)\max\{\sigma,\kappa\}^2+B(m)\left(\frac{k}{\mu_\phi(M)}\right)^{\frac{2}{m}}.\]
This completes the proof.
\end{proof}

\end{document}